\documentclass[12p]{amsart}
\usepackage{amssymb}
\usepackage{amsmath}
\usepackage{amsfonts}
\usepackage{geometry}
\usepackage{graphicx}
\usepackage{mathrsfs,amssymb}
\usepackage[pagewise]{lineno}

\usepackage{hyperref}
\usepackage{cleveref}
\usepackage{slashed}

\theoremstyle{plain}

\newtheorem{definition}{Definition}

\newtheorem{lemma}{Lemma}

\newtheorem{proposition}{Proposition}
\newtheorem{remark}{Remark}

\newtheorem{theorem}{Theorem}
\numberwithin{equation}{section}

\def\XXint#1#2#3{{\setbox0=\hbox{$#1{#2#3}{\int}$ }
\vcenter{\hbox{$#2#3$ }}\kern-.6\wd0}}

\begin{document}
\title[]{A Liouville theorem for the Chern--Simons--Schr{\"o}dinger equation}

\author{Benjamin Dodson}

\begin{abstract}
In this paper we prove a Liouville theorem for the Chern--Simons--Schr{\"o}dinger equation. This result is consistent with the soliton resolution conjecture for initial data that does not lie in a weighted space. See \cite{kim2022soliton} for the soliton resolution result in a weighted space.
\end{abstract}
\maketitle

\section{Introduction}
The self-dual Chern--Simons--Schr{\"o}dinger equation with $m$-equivariance is
\begin{equation}\label{1.1}
i(\partial_{t} + i A_{t}[u]) u + \partial_{r}^{2} u + \frac{1}{r} \partial_{r} u - (\frac{m + A_{\theta}[u]}{r})^{2} u + |u|^{2} u = 0,
\end{equation}
where
\begin{equation}\label{1.2}
A_{t}[u] = -\int_{r}^{\infty} (m + A_{\theta}[u]) |u|^{2} \frac{dr'}{r'}, \qquad \text{and} \qquad A_{\theta}[u] = -\frac{1}{2} \int_{0}^{r} |u|^{2} r' dr'.
\end{equation}

The Chern--Simons--Schr{\"o}dinger equation was introduced in \cite{jackiw1990classical} as a nonrelativistic planar quantum electromagnetic model that exhibits self-duality. It is a gauge covariant nonlinear Schr{\"o}dinger equation on $\mathbb{R}^{2}$. See also \cite{dunne2009self}, \cite{jackiw1990soliton}, \cite{jackiw1991time}. The model $(\ref{1.1})$ is derived after fixing the Coulomb gauge condition and imposing the equivariant symmetry on the scalar field $\phi$:
\begin{equation}\label{1.2.1}
\phi(t,x) = u(t, r) e^{im \theta}.
\end{equation}
See \cite{kim2020blow}, \cite{kim2019pseudoconformal}, and \cite{kim2020construction}.

\begin{remark}
The non-equivariant Chern--Simons--Schr{\"o}dinger equation will not be discussed here. See \cite{berge1995blowing}, \cite{huh2013energy}, \cite{lim2018large}, \cite{liu2014local}, and \cite{oh2015decay} for more information.
\end{remark}

The solution to $(\ref{1.1})$ has the conserved quantities mass and energy.
\begin{equation}\label{1.3}
\aligned
E[u] &= \int \frac{1}{2} |\partial_{r} u|^{2} + \frac{1}{2} (\frac{m + A_{\theta}[u]}{r})^{2} |u|^{2} - \frac{1}{4} |u|^{4} dx, \\
M[u] &= \int |u|^{2} dx,
\endaligned
\end{equation}
where we denote $\int f(r) = 2 \pi \int f(r) r dr$. Indeed, $(\ref{1.1})$ is the Hamiltonian PDE for the energy in $(\ref{1.3})$. We also use the inner product
\begin{equation}\label{1.3.1}
\langle f, g \rangle = Re \int f(r) \overline{g(r)} r dr.
\end{equation}
 Integrating by parts,
\begin{equation}\label{1.4}
\aligned
\frac{d}{dt} E[u] = \int Re (\partial_{r} \bar{u})(\partial_{r} u_{t}) r dr + \int (\frac{m + A_{\theta}[u]}{r})^{2} Re(\bar{u} u_{t}) r dr \\ - \int |u|^{2} Re(\bar{u} u_{t}) r dr - \int \frac{m + A_{\theta}[u]}{r} |u|^{2} \int_{0}^{r} Re(\bar{u} u_{t}) r' dr' dr \\
= -\langle (\partial_{rr} + \frac{1}{r} \partial_{r}) \bar{u} , u_{t} \rangle + \langle (\frac{m + A_{\theta}[u]}{r})^{2} u, u_{t} \rangle - \langle |u|^{2} u, u_{t} \rangle - \int_{0}^{\infty} |u|^{2} \int_{r}^{\infty} \frac{m + A_{\theta}[u]}{r} Re(\bar{u} u_{t}) dr r' dr' \\
= -\langle (\partial_{rr} + \frac{1}{r} \partial_{r}) \bar{u} , u_{t} \rangle + \langle (\frac{m + A_{\theta}[u]}{r})^{2} u, u_{t} \rangle - \langle |u|^{2} u, u_{t} \rangle + \langle A_{t}[u] \bar{u}, u_{t} \rangle = \langle i u_{t}, u_{t} \rangle = 0.
\endaligned
\end{equation}
Thus,
\begin{equation}\label{1.4.1}
\partial_{t} u = -i \nabla E[u],
\end{equation}
where $\nabla$ (acting on a functional) is the Frechet derivative with respect to the inner product $\langle \cdot, \cdot \rangle$. Also,
\begin{equation}\label{1.4.2}
\partial_{t} M[u] = Re \int \bar{u} u_{t} = Re \int \bar{u} (i \partial_{r}^{2} u + \frac{i}{r} \partial_{r} u) = 0.
\end{equation}

The energy functional can be written in the self-dual form
\begin{equation}\label{1.4.3}
E[u] = \int \frac{1}{2} |D_{u} u|^{2},
\end{equation}
where $D_{u}$ is the covariant Cauchy--Riemann operator defined by
\begin{equation}\label{1.4.4}
D_{u} f = \partial_{r} f - \frac{m + A_{\theta}[u]}{r} f.
\end{equation}
Indeed,
\begin{equation}\label{1.4.5}
-Re \int (\partial_{r} \bar{f}) \cdot (\frac{m + A_{\theta}[u]}{r}) f(r) r dr = -\frac{1}{2} \int (m + A_{\theta}[u]) \partial_{r}(|f|^{2}) dr = \frac{1}{2} \int \partial_{r} A_{\theta}[u] |f|^{2} = -\frac{1}{4} \int_{0}^{\infty} |f|^{4} r dr.
\end{equation}

\begin{definition}[Bogomol'nyi operator]\label{d1.1}
The operator $u \mapsto D_{u}$ is called the Bogomol'nyi operator. Due to $(\ref{1.4.3})$ and the Hamiltonian structure, any static solutions to $(\ref{1.1})$ are given by solutions to the Bogomol'nyi equation
\begin{equation}\label{1.4.6}
D_{Q} Q = 0.
\end{equation}
For $m \geq 0$, there is an explicit $m$-equivariant static solution to the Bogomol'nyi equation, the Jackiw--Pi vortex. This solution is unique up to the symmetries of the equation (\cite{li2020threshold}):
\begin{equation}\label{1.5}
Q(r) = \sqrt{8} (m + 1) \frac{r^{m}}{1 + r^{2m + 2}}, \qquad m \geq 0.
\end{equation}
\end{definition}

Equation $(\ref{1.1})$ has the pseudoconformal transform $\mathcal C$,
\begin{equation}\label{1.6}
[\mathcal C u](t, r) = \frac{1}{|t|} u(-\frac{1}{t}, \frac{r}{|t|}) e^{i r^{2}/4t}, \qquad \forall t \neq 0.
\end{equation}
Since the soliton solution is non-scattering, applying the pseudoconformal transform to $Q$ gives an explicit, finite-time blowup solution,
\begin{equation}\label{1.7}
S(t, r) = \frac{1}{|t|} Q(\frac{r}{|t|}) e^{-i \frac{r^{2}}{4t}}, \qquad t < 0.
\end{equation}

It is conjectured that any blowup solution must contain either $(\ref{1.5})$ or $(\ref{1.7})$. Indeed, \cite{liu2016global} proved global well-posedness and scattering for $(\ref{1.1})$ with initial data with mass below the mass of the ground state,
\begin{equation}\label{1.8}
\| u_{0} \|_{L^{2}} < \| Q \|_{L^{2}}.
\end{equation}
\begin{theorem}\label{t1.2}
Let $m \in \mathbb{Z}_{+}$. Let $\phi_{0} \in L_{m}^{2}$ with $\| \phi_{0} \|_{L_{m}^{2}}$ and
\begin{equation}\label{1.9}
\| u_{0} \|_{L^{2}}^{2} < 8 \pi (m + 1).
\end{equation}
Then $(\ref{1.1})$ is globally well-posed in $L_{m}^{2}$ and scatters both forward and backward in time.
\end{theorem}
\begin{proof}
See Theorem $1.3$ of \cite{liu2016global}.
\end{proof}

Making a $u$-substitution,
\begin{equation}\label{4.9}
\| Q \|_{L^{2}}^{2} = 16 \pi (m + 1)^{2} \int_{0}^{\infty} \frac{r^{2m + 1}}{(1 + r^{2m + 2})^{2}} dr = 8 \pi (m + 1) \int_{0}^{\infty} \frac{du}{(1 + u)^{2}} = 8 \pi (m + 1).
\end{equation}
\begin{remark}
A function $u_{0} \in L_{m}^{2}$ if $u_{0} \in L^{2}$ and $u_{0}$ has the form $(\ref{1.2.1})$. A function $u_{0} \in H_{m}^{1}$ if $u_{0} \in H^{1}$ and has the form $(\ref{1.2.1})$.
\end{remark}

More recently, \cite{kim2020blow} and \cite{kim2022soliton} proved a decomposition for finite time blowup solutions to $(\ref{1.1})$ with finite energy and initial data in a weighted Sobolev space.
\begin{theorem}\label{t1.3}
If $m \in \mathbb{Z}_{+}$ and $u$ is a $ H_{m}^{1}$-solution to $(\ref{1.1})$ that blows up forward in time at $T < +\infty$, then $u(t)$ admits the decomposition
\begin{equation}\label{1.10}
u(t, \cdot) - Q_{\lambda(t), \gamma(t)} \rightarrow z^{\ast}, \qquad \text{in} \qquad L^{2}, \qquad \text{as} \qquad t \nearrow T.
\end{equation}

Moreover, using the pseudoconformal transformation in $(\ref{1.6})$, it is possible to obtain a similar decomposition for a solution that exists globally forward in time, but fails to scatter forward in time, for initial data that also lies in a weighted $L^{2}$-space.
\end{theorem}
\begin{proof}
See \cite{kim2022soliton}.
\end{proof}

In this paper, we prove a Liouville theorem for solutions to $(\ref{1.1})$ that are global in at least one time direction.
\begin{theorem}[Liouville theorem]\label{t1.4}
Suppose $u_{0} \in H_{m}^{1}$ is an initial data for $(\ref{1.1})$ that has a solution on the maximal interval of existence $I$. Furthermore, suppose that $I = (-\infty, t_{0})$, where $t_{0}$ could be $+\infty$, or $(t_{0}, \infty)$, where $t_{0}$ could be $-\infty$. Also suppose that for any $\eta > 0$, there exists $R(\eta) < \infty$ such that
\begin{equation}\label{1.11}
\sup_{t \in I} \int_{|x| \geq R} |u(t, x)|^{2} dx < \eta,
\end{equation}
where $I$ is the interval of existence for a solution to $(\ref{1.1})$. Then $u$ is equal to the soliton solution $(\ref{1.5})$, up to the scaling symmetry,
\begin{equation}\label{1.12}
u_{\lambda}(t, r) = \frac{1}{\lambda} u(\frac{t}{\lambda^{2}}, \frac{x}{\lambda}), \qquad \lambda > 0,
\end{equation}
and multiplication by $e^{i \gamma}$ for some $\gamma \in \mathbb{R}$.
\end{theorem}

This result was inspired by the Liouville theorem of \cite{martel2000liouville}. There, \cite{martel2000liouville} proved that for a solution to the mass-critical generalized Korteweg de-Vries equation,
\begin{equation}\label{1.13}
u_{t} + u_{xxx} + \partial_{x}(u^{5}) = 0,
\end{equation}
with initial data close to the rescaled soliton in $H_{x}^{1}(\mathbb{R})$, and with $H^{1}(\mathbb{R})$ norm uniformly bounded, then the solution to $(\ref{1.13})$ must be the soliton. For the mass-critical generalized KdV equation, it is expected that multi-soliton solutions occur, which necessitates additional constraints on the size of the initial data than we have here.\medskip

Unlike the generalized KdV equation, the structure of the self-dual Chern--Simons--Schr{\"o}dinger equation is defocusing outside of a soliton.
For this reason, it is unnecessary to require a uniform bound on $\| u(t) \|_{H^{1}}$ on $I$. Also, since $u_{0}$ need not be close to the soliton, we do not assume an a priori bound on $\| u_{0} \|_{L^{2}}$. 

\begin{remark}
For a solution to $(\ref{1.1})$, $u \in H_{m}^{1, 1}$, that exists globally forward in time, then either $u(t)$ scatters forward in time, or $u(t)$ admits the decomposition
\begin{equation}\label{1.13.1}
u(t, \cdot) - Q_{\lambda(t), \gamma(t)} - e^{it \Delta^{(-m - 2)}} u^{\ast} \rightarrow 0, \qquad \text{in} \qquad L^{2}, \qquad \text{as} \qquad t \rightarrow \infty.
\end{equation}
Here, $e^{it \Delta^{(-m - 2)}} u^{\ast}$ is the solution to the free, $(-m - 2)$--equivariant Schr{\"o}dinger flow,
\begin{equation}\label{1.13.2}
i \partial_{t} u + \partial_{r}^{2} u + \frac{1}{r} \partial_{r} u - \frac{(m + 2)^{2}}{r^{2}} u = 0.
\end{equation}
The space $H_{m}^{1, 1}$ is the space of $m$-equivariant functions, $(\ref{1.2.1})$, that lie in $H^{1}$ and the weighted $L^{2}$-space, $\| |x| u \|_{L^{2}} < \infty$.\medskip

If $(\ref{1.13.1})$ could be proved for any $u_{0} \in  H_{m}^{1}$, then Theorem $\ref{t1.4}$ would likely follow fairly easily, since $(\ref{1.13.1})$ would at least imply that $u^{\ast} = 0$, and thus $\| u \|_{L^{2}} = \| Q \|_{L^{2}}$. This is due to the fact that a scattering solution, or a solution with a scattering piece could not satisfy $(\ref{1.11})$.
\end{remark}

The proof of Theorem $\ref{t1.4}$ may be broken down into three steps. First, using a virial identity combined with $(\ref{1.11})$, we prove that any solution to $(\ref{1.1})$ that satisfies $(\ref{1.11})$ must have the mass of the soliton,
\begin{equation}\label{1.14}
\| u(t) \|_{L^{2}} = \| Q \|_{L^{2}}.
\end{equation}

Next, using an argument analogous to the argument in \cite{merle1993determination}, we prove that a solution to $(\ref{1.1})$ satisfying $(\ref{1.11})$ and $(\ref{1.14})$ must be global in both time directions. Indeed, any finite time blowup solution with $\| u_{0} \|_{L^{2}} = \| Q \|_{L^{2}}$ must be a rescaled version of $(\ref{1.7})$, which clearly does not satisfy $(\ref{1.11})$.

Combining $(\ref{1.11})$, $(\ref{1.14})$, $u_{0} \in H^{1}$, and the virial identity shows that $u(t)$ is the soliton.

\section{Mass above the soliton}
In this section, we prove that if $(\ref{1.11})$, then $u$ has the same mass as the soliton.

\begin{theorem}\label{t2.1}
Suppose $u_{0} \in H^{1}$ is an initial data for $(\ref{1.1})$ that has a solution on the maximal interval of existence $I$. Also suppose that for any $\eta > 0$, there exists $R(\eta) < \infty$ such that
\begin{equation}\label{2.1}
\sup_{t \in I} \int_{|x| \geq R} |u(t, x)|^{2} dx < \eta,
\end{equation}
where $I$ is the interval of existence for a solution to $(\ref{1.1})$. Then $\| u \|_{L^{2}} = \| Q \|_{L^{2}}$, where $Q$ is the soliton, $(\ref{1.5})$.
\end{theorem}

\begin{proof}
We prove this using the virial identity
\begin{equation}\label{2.2}
\partial_{t} \int Im(\bar{u} \cdot r \partial_{r} u) = 4 E[u].
\end{equation}

\begin{lemma}\label{l2.2}
For any solution $u$, $0 < R < \infty$,
\begin{equation}\label{2.3}
\sup_{t \in I} \int \psi(\frac{r}{R}) Im[\bar{u} \cdot r \partial_{r} u] \lesssim_{E[u], M[u]} R.
\end{equation}
Here $\psi(r) \in C^{\infty}(\mathbb{R}^{2})$ is a radially symmetric function, $\psi(r) = 1$ for $r \leq 1$, $\psi(r) = \frac{3}{2r}$ for $r \geq 2$. Moreover,
\begin{equation}
\partial_{r} (\psi(r) r) = \phi(r),
\end{equation}
where $\phi(r)$ is a positive, smooth function, $\phi(r) = 1$ for $r \leq 1$, $\phi(r)$ supported on $r \leq 2$, and $\phi(r) = \chi(r)^{2}$ for some $\chi \in C_{0}^{\infty}(\mathbb{R}^{2})$.
\end{lemma}
\begin{proof}[Proof of Lemma]
Consider two cases separately, when $\| u(t) \|_{H^{1}}$ is uniformly bounded, and the case when $\| u(t) \|_{H^{1}}$ is not uniformly bounded.\medskip

\noindent \emph{Case 1:}
\begin{equation}\label{2.4}
\sup_{t \in I} \| \nabla u(t) \|_{L^{2}} < \infty.
\end{equation}
In this case, $I = \mathbb{R}$. Now then,
\begin{equation}\label{2.5}
\int \psi(\frac{r}{R}) Im[\bar{u} \cdot r \partial_{r} u] \lesssim \sup_{t \in I} \| \nabla u \|_{L^{2}} \| \psi(\frac{r}{R}) ru \|_{L^{2}} \lesssim_{M[u]} R \sup_{t \in I} \| \nabla u(t) \|_{L^{2}(\mathbb{R}^{2})} \lesssim R.
\end{equation}

\noindent \emph{Case 2:} Since $\| \nabla u(t) \|_{L^{2}}$ is continuous in time, if $\sup_{t \in I} \| \nabla u(t) \|_{L^{2}} = \infty$, then there exists a sequence $t_{n}$ such that
\begin{equation}\label{2.6}
\| \nabla u(t_{n}) \|_{L^{2}} = n.
\end{equation}
Set
\begin{equation}\label{2.7}
\lambda(t_{n}) = \frac{\| \nabla u(t_{n}) \|_{L^{2}}}{\| \nabla Q \|_{L^{2}}}.
\end{equation}
Plugging $\lambda(t_{n})$ into $(\ref{1.12})$, let
\begin{equation}\label{2.8}
v(t_{n}, x) = \frac{1}{\lambda(t_{n})} u(t_{n}, \frac{x}{\lambda(t_{n})}).
\end{equation}
By direct computation,
\begin{equation}\label{2.9}
E[v(t_{n}, x)] = \frac{1}{\lambda(t_{n})^{2}} E[u(t_{n})] = \frac{1}{\lambda(t_{n})^{2}} E[u_{0}].
\end{equation}

Now, recall Proposition $4.1$ from \cite{kim2022soliton}.
\begin{proposition}[Decomposition]\label{p2.3}
Let $Z_{1}, Z_{2} \in C_{c, m}^{\infty}$ be profiles that satisfy
\begin{equation}\label{2.10}
det \begin{pmatrix} (\Lambda Q, Z_{1})_{r} & (i Q, Z_{1})_{r} \\ (\Lambda Q, Z_{2})_{r} & (i Q, Z_{2})_{r} \end{pmatrix} \neq 0.
\end{equation}
Here, $\Lambda$ is the operator $r \partial_{r} + 1$. Then for any $M < \infty$, there exists $0 < \alpha^{\ast} \ll 1$ such that the following properties hold for all $u \in H_{m}^{1}$ with $\| u \|_{L^{2}} \leq M$ satisfying the small energy condition $\sqrt{E[u]} \leq \alpha^{\ast} \| u \|_{H_{m}^{1}}$.\medskip

There exists a unique $(\lambda, \gamma) \in \mathbb{R}_{+} \times \mathbb{R} / 2 \pi \mathbb{Z}$ such that $\epsilon \in H_{m}^{1}$, defined by the relation
\begin{equation}\label{2.11}
u = [Q + \epsilon]_{\lambda, \gamma},
\end{equation}
satisfies the orthogonality conditions,
\begin{equation}\label{2.12}
(\epsilon, Z_{1})_{r} = (\epsilon, Z_{2})_{r} = 0,
\end{equation}
and smallness
\begin{equation}\label{2.13}
\| \epsilon \|_{\dot{\mathcal H}_{m}^{1}} \sim_{M} \lambda \sqrt{E[u]}.
\end{equation}
\end{proposition}
\begin{remark}
The space $\mathcal H_{m}^{1}$ is a function space adapted to the linear coercivity of the energy. When $m \geq 1$, as is true in this paper, the spaces $\mathcal H_{m}^{1}$ and $H_{m}^{1}$ are equivalent.
\end{remark}

\begin{proof}
The proof in  \cite{kim2022soliton} relies on the uniqueness of the soliton as a function with zero energy, the nonlinear coercivity of energy in  \cite{kim2022soliton}, and the implicit function theorem.

\begin{proposition}[Nonlinear coercivity of energy]\label{p2.4}
For any $M > 0$, there exists $\eta > 0$ such that the nonlinear coercivity
\begin{equation}\label{2.13.1}
E[Q + \epsilon] \sim_{M} \| \epsilon \|_{\dot{\mathcal H}_{m}^{1}}^{2},
\end{equation}
holds for any $\epsilon \in H_{m}^{1}$ with $\| \epsilon \|_{L^{2}} \leq M$ satisfying the orthogonality conditions $(\ref{2.12})$ and smallness $\| \epsilon \|_{\dot{\mathcal H}_{m}^{1}} \leq \eta$.
\end{proposition}

\begin{proof}[Proof of Proposition $\ref{p2.4}$]
We follow the argument in \cite{kim2022soliton}. Observe that, by $(\ref{1.4.4})$,
\begin{equation}\label{2.13.2}
2 E[Q + \epsilon] = \| D_{Q + \epsilon}(Q + \epsilon) \|_{L^{2}}^{2} = \| \partial_{r} (Q + \epsilon) - \frac{m + A_{\theta}[Q + \epsilon]}{r} (Q + \epsilon) \|_{L^{2}}^{2}.
\end{equation}
Now then, since $D_{Q} Q = 0$,
\begin{equation}\label{2.13.3}
\aligned
\partial_{r} (Q + \epsilon) - \frac{m + A_{\theta}[Q + \epsilon]}{r} (Q + \epsilon) = \partial_{r} Q - \frac{m + A_{\theta}[Q]}{r} Q + \partial_{r} \epsilon - \frac{m + A_{\theta}[Q]}{r} \epsilon \\ - \frac{2 A_{\theta}[Q, \epsilon]}{r} Q - \frac{A_{\theta}[\epsilon]}{r} Q
- \frac{2 A_{\theta}[Q, \epsilon]}{r} \epsilon - \frac{A_{\theta}[\epsilon]}{r} \epsilon \\
= D_{Q} \epsilon - \frac{2 A_{\theta}[Q, \epsilon]}{r} Q - \frac{A_{\theta}[\epsilon]}{r} Q
- \frac{2 A_{\theta}[Q, \epsilon]}{r} \epsilon - \frac{A_{\theta}[\epsilon]}{r} \epsilon = L_{Q} \epsilon - \frac{A_{\theta}[\epsilon]}{r} Q - \frac{2 A_{\theta}[Q, \epsilon]}{r} \epsilon - \frac{A_{\theta}[\epsilon]}{r} \epsilon,
\endaligned
\end{equation}
Here,
\begin{equation}\label{2.13.4}
A_{\theta}[\psi_{1}, \psi_{2}] = -\frac{1}{2} \int_{0}^{r} Re(\bar{\psi_{1}} \psi_{2}) r' dr',
\end{equation}
and
\begin{equation}\label{2.13.5}
L_{Q} \epsilon = D_{Q} \epsilon - \frac{2 A_{\theta}[Q, \epsilon]}{r} Q.
\end{equation}
Using the coercivity of $L_{Q}$ proved in \cite{kim2020blow}, \cite{kim2019pseudoconformal},
\begin{lemma}[Coercivity of $L_{Q}$]\label{l2.4.1}
Let $m \geq 0$. Let $\mathcal Z_{1}$, $\mathcal Z_{2} \in C_{c, m}^{\infty}$ satisfy $(\ref{2.10})$. Then,
\begin{equation}\label{2.13.6}
\| L_{Q} f \|_{L^{2}} \sim \| f \|_{\dot{\mathcal H}_{m}^{1}}, \qquad \forall f \in \dot{\mathcal H}_{m}^{1} \qquad \text{with} \qquad (f, \mathcal Z_{1}) = (f, \mathcal Z_{2}) = 0.
\end{equation}
\end{lemma}

Now then, by Hardy's inequality,
\begin{equation}\label{2.13.7}
\| \frac{2}{r} A_{\theta}[Q, \epsilon] \epsilon \|_{L^{2}} \lesssim (\int_{0}^{\infty} Q \langle r \rangle |\epsilon| dr) \cdot \| \frac{1}{\langle r \rangle} \epsilon \|_{L^{2}} \lesssim \| \epsilon \|_{\dot{\mathcal H}_{m}^{1}}^{3/2} \| \epsilon \|_{L^{2}}^{1/2} \lesssim_{M} \| \epsilon \|_{\dot{\mathcal H}_{m}^{1}}^{3/2}.
\end{equation}
Also, by Hardy's inequality,
\begin{equation}\label{2.13.8}
\| \frac{1}{r} A_{\theta}[\epsilon] Q \|_{L^{2}} \lesssim (\int_{0}^{\infty} \frac{1}{\langle r \rangle^{1/2}} |\epsilon|^{2} dr) \cdot \| \langle r \rangle^{1/2} Q \|_{L^{2}} \lesssim \| \epsilon \|_{L^{2}}^{1/2} \| \epsilon \|_{\dot{\mathcal H}_{m}^{1}}^{3/2} \lesssim_{M} \| \epsilon \|_{\dot{\mathcal H}_{m}^{1}}^{3/2}.
\end{equation}
Therefore, we have proved
\begin{equation}\label{2.13.9}
2 E[Q + \epsilon] = \| L_{Q} \epsilon - \frac{A_{\theta}[\epsilon]}{r} \epsilon \|_{L^{2}}^{2} + O_{M}(\| \epsilon \|_{\dot{H}_{m}^{1}}^{3}).
\end{equation}
Now decompose $\epsilon = \chi_{R} \epsilon + (1 - \chi_{R}) \epsilon$, where $\chi_{R}(r) = \chi(\frac{r}{R})$ is the function defined in Lemma $\ref{l2.2}$. Then decompose
\begin{equation}\label{2.13.10}
L_{Q} \epsilon - \frac{A_{\theta}[\epsilon]}{r} \epsilon = L_{Q} \epsilon - \frac{A_{\theta}[\epsilon]}{r} (1 - \chi_{R}) \epsilon - \frac{A_{\theta}[\epsilon]}{r} \chi_{R} \epsilon.
\end{equation}
By direct computation,
\begin{equation}\label{2.13.11}
\| \frac{A_{\theta}[\epsilon]}{r} \chi_{R} \epsilon \|_{L^{2}} \lesssim (\int_{0}^{R} |\epsilon|^{2} r dr) \| \frac{\epsilon}{r} \|_{L^{2}} \lesssim R \| \epsilon \|_{L^{2}} \| \epsilon \|_{\dot{\mathcal H}_{m}^{1}}^{2} \lesssim RM || \epsilon \|_{\dot{\mathcal H}_{m}^{1}}^{2}.
\end{equation}
Next, decompose
\begin{equation}\label{2.13.12}
L_{Q} \epsilon = D_{Q} (\chi_{R} \epsilon) + D_{Q}((1 - \chi_{R}) \epsilon) - \frac{2 A_{\theta}[Q, \chi_{R} \epsilon]}{r} Q - \frac{2 A_{\theta}[Q, (1 - \chi_{R}) \epsilon]}{r} Q.
\end{equation}
Using the decay of $Q$,
\begin{equation}\label{2.13.13}
\| \frac{2 A_{\theta}[Q, (1 - \chi_{R}) \epsilon]}{r} Q \|_{L^{2}} \lesssim \frac{1}{R} \| \epsilon \|_{\dot{\mathcal H}_{m}^{1}}.
\end{equation}
Therefore,
\begin{equation}\label{2.13.14}
\| L_{Q} \epsilon - \frac{A_{\theta}[\epsilon]}{r} \epsilon \|_{L^{2}} = \| L_{Q}(\chi_{R} \epsilon) + (D_{Q} - \frac{A_{\theta}[\epsilon]}{r})(1 - \chi_{R}) \epsilon \|_{L^{2}} + \frac{1}{R} \| \epsilon \|_{\dot{\mathcal H}_{m}^{1}} + RM \| \epsilon \|_{\dot{\mathcal H}_{m}^{1}}^{2}.
\end{equation}
Decompose
\begin{equation}\label{2.13.15}
\aligned
\| L_{Q} (\chi_{R} \epsilon) + (D_{Q} - \frac{A_{\theta}[\epsilon]}{r}) (1 - \chi_{R}) \epsilon \|_{L^{2}}^{2} = \| L_{Q} (\chi_{R} \epsilon) \|_{L^{2}}^{2} \\ + \| (D_{Q} - \frac{A_{\theta}[\epsilon]}{r}) (1 - \chi_{R}) \epsilon \|_{L^{2}}^{2} + 2 \langle L_{Q} (\chi_{R} \epsilon), (D_{Q} - \frac{A_{\theta}[\epsilon]}{r}) (1 - \chi_{R}) \epsilon \rangle.
\endaligned
\end{equation}
By the support of $\chi_{R}$ and $1 - \chi_{R}$,
\begin{equation}\label{2.13.16}
 \langle L_{Q} (\chi_{R} \epsilon), (D_{Q} - \frac{A_{\theta}[\epsilon]}{r}) (1 - \chi_{R}) \epsilon \rangle \lesssim \| |\partial_{r} \epsilon| + \frac{1}{r} \epsilon \|_{L^{2}(\frac{R}{2} \leq r \leq R)}^{2}.
 \end{equation}
 Using the nonlinear Hardy inequality in \cite{kim2022soliton},
\begin{equation}\label{2.13.17}
\| (D_{Q} - \frac{A_{\theta}[\epsilon]}{r}) (1 - \chi_{R}) \epsilon \|_{L^{2}}^{2} \sim_{M} \| (1 - \chi_{R}) \epsilon \|_{\dot{\mathcal H}_{m}^{1}}^{2}.
\end{equation}
Therefore, we have proved
\begin{equation}\label{2.13.18}
2 E[Q + \epsilon] \sim \| \chi_{R} \epsilon \|_{\dot{\mathcal H}_{m}^{1}}^{2} + \| (1 - \chi_{R}) \epsilon \|_{\dot{\mathcal H}_{m}^{1}}^{2} + R^{2} M^{2} \| \epsilon \|_{\dot{\mathcal H}_{m}^{1}}^{4} + \| \epsilon \|_{\dot{\mathcal H}_{m}^{1}}^{3} + \frac{1}{R} \| \epsilon \|_{\dot{\mathcal H}_{m}^{1}} +  \| |\partial_{r} \epsilon| + \frac{1}{r} \epsilon \|_{L^{2}(\frac{R}{2} \leq r \leq R)}^{2}.
\end{equation}
Taking $R \ll \frac{1}{\| \epsilon \|_{\dot{H}_{m}^{1}}}$ and averaging over $\log(\| \epsilon \|_{\dot{\mathcal H}_{m}^{1}})$ intervals of the form $\frac{R}{2} \leq r \leq R$, the proof of Proposition $\ref{p2.4}$ is complete.
\end{proof}

Now then, suppose there exists a sequence $\| v_{n} \|_{\mathcal H_{m}^{1}}$ constant, $E[v_{n}] \rightarrow 0$. By the uniqueness of $Q$ (up to scaling) as a solution to $E[Q] = 0$ and the fact that $E[u] \geq 0$, and thus $Q$ is an energy minimizer, $v_{n}$ converges in $\mathcal H_{m}^{1}$, and thus, $v_{n} \rightarrow Q$ in $\mathcal H_{m}^{1}$. Therefore, $\| v_{n} - Q \|_{\dot{\mathcal H}_{m}^{1}} \rightarrow 0$ as $n \rightarrow \infty$.

Next, since
\begin{equation}\label{2.13.19}
\frac{\partial}{\partial \lambda} \lambda Q(\frac{x}{\lambda})|_{\lambda = 1} = \Lambda Q,
\end{equation}
and
\begin{equation}\label{2.13.20}
\frac{\partial}{\partial \gamma} e^{i \gamma} Q|_{\gamma = 0} = i Q,
\end{equation}
combining the implicit function theorem with $(\ref{2.10})$, for $E[v_{n}]$ sufficiently small, we can find a unique $\lambda$ and $\gamma$ such that $(\ref{2.12})$ holds.

Then by Proposition $\ref{p2.4}$, $(\ref{2.13})$ holds.
\end{proof}

%Furthermore, from \cite{kim2020blow}, for $Z_{1}, Z_{2} \in C_{c, m}^{\infty}$ that satisfy $(\ref{2.10})$,
%\begin{equation}\label{2.13.1}
%\| L_{Q} f \|_{L^{2}} \sim \| f \|_{\dot{\mathcal H}_{m}^{1}}, \qquad \forall f \in \dot{\mathcal H}_{m}^{1}, \qquad \text{with} \qquad (f, Z_{1})_{r} = (f, Z_{2})_{r} = 0.
%\end{equation}

Therefore, for $n(M)$ sufficiently large,
\begin{equation}\label{2.14}
\sqrt{E[v](t_{n})} \leq \alpha^{\ast} \| v(t_{n}) \|_{\dot{H}_{m}^{1}} \sim \alpha^{\ast},
\end{equation}
and we can make the decomposition
\begin{equation}\label{2.15}
v(t_{n}) = [Q + \epsilon]_{\lambda(t_{n}), \gamma(t_{n})}.
\end{equation}
Now, note that $(\ref{1.12})$ implies
\begin{equation}\label{2.16}
\int \psi(\frac{x}{R}) Im[\bar{u} \cdot r \partial_{r} u](t_{n}) = \int \psi(\frac{x \lambda(t_{n})}{R}) Im[\bar{v} \cdot r \partial_{r} v](t_{n})
\end{equation}
\begin{equation}\label{2.17}
= \int \psi(\frac{x \lambda(t_{n})}{R}) Im[\bar{Q} \cdot r \partial_{r} Q]
\end{equation}
\begin{equation}\label{2.18}
+ \int \psi(\frac{x \lambda(t_{n})}{R}) Im[\bar{Q} \cdot r \partial_{r} \epsilon](t_{n}) + \int \psi(\frac{x \lambda(t_{n})}{R}) Im[\bar{\epsilon} \cdot r \partial_{r} Q](t_{n})
\end{equation}
\begin{equation}\label{2.19}
+ \int \psi(\frac{x \lambda(t_{n})}{R}) Im[\bar{\epsilon} \cdot r \partial_{r} \epsilon](t_{n}).
\end{equation}

Since $Q$ is real-valued, $(\ref{2.17}) = 0$. Next,
\begin{equation}\label{2.20}
(\ref{2.19}) \lesssim \frac{R}{\lambda(t_{n})} \| \nabla \epsilon_{\lambda, \gamma} \|_{L^{2}} \lesssim_{M[u]} \frac{R}{\lambda(t_{n})} \sqrt{E[v(t_{n})]} \lesssim_{M[u]} R \sqrt{E[u_{0}]}.
\end{equation}
Next,
\begin{equation}\label{2.21}
\int \psi(\frac{x \lambda(t_{n})}{R}) Im[\bar{\epsilon} \cdot r \partial_{r} Q] \lesssim \| \epsilon_{\lambda, \gamma} \|_{L^{2}} \| r \partial_{r} Q_{\lambda, \gamma} \|_{L^{2}} \lesssim 1.
\end{equation}
Finally, integrating by parts,
\begin{equation}\label{2.22}
\aligned
\int \psi(\frac{x \lambda(t_{n})}{R}) Im[\bar{Q}_{\lambda, \gamma} \cdot r \partial_{r} \epsilon_{\lambda, \gamma}] = -\int \psi(\frac{x \lambda(t_{n})}{R}) Im[r \partial_{r} \bar{Q}_{\lambda, \gamma} \cdot \epsilon_{\lambda, \gamma}] \\ - \int \partial_{r}(\psi(\frac{r \lambda(t_{n})}{R}) r) Im[\bar{Q}_{\lambda, \gamma} \cdot \epsilon_{\lambda, \gamma}] \lesssim \| \epsilon \|_{L^{2}} \lesssim 1.
\endaligned
\end{equation}
This proves the Lemma.
\end{proof}

Now then, compute
\begin{equation}\label{2.23}
\aligned
\frac{d}{dt} M(t) = \int \psi(\frac{r}{R}) Re[\bar{u} \cdot r \partial_{r} \Delta u] - \int \psi(\frac{r}{R}) Re[\Delta \bar{u} \cdot r \partial_{r} u] \\
- \int \psi(\frac{r}{R}) Re[\bar{u} \cdot r \partial_{r}(A_{t}[u] u)] + \int \psi(\frac{r}{R}) Re[A_{t}[u] \bar{u} \cdot r \partial_{r} u] \\
-\int \psi(\frac{r}{R}) Re[\bar{u} \cdot r \partial_{r}((\frac{m + A_{\theta}[u]}{r})^{2} u)] + \int \psi(\frac{r}{R}) Re[(\frac{m + A_{\theta}[u]}{r})^{2} \bar{u} \cdot r \partial_{r} u] \\
\int \psi(\frac{r}{R}) Re[\bar{u} \cdot r \partial_{r}(|u|^{2} u)] - \int \psi(\frac{r}{R}) Re[|u|^{2} \bar{u} \cdot r \partial_{r} u]
\endaligned
\end{equation}

\begin{equation}\label{2.24}
\aligned
= \int \phi(\frac{r}{R}) |\partial_{r} u|^{2} - \frac{1}{4} \int \frac{1}{R^{2}} \phi''(\frac{r}{R}) |u|^{2} - \int \psi(\frac{r}{R}) |u|^{2} (m + A_{\theta}[u]) + \int \psi(\frac{r}{R}) |u|^{2} r (\frac{m + A_{\theta}[u]}{r}) \\
+ 2 \int \psi(\frac{r}{R}) |u|^{2} (\frac{m + A_{\theta}[u]}{r})^{2}  - \frac{1}{4} \int \phi(\frac{r}{R}) |u|^{4} - \int \psi(\frac{r}{R}) |u|^{4} \\
= 2 \int \phi(\frac{r}{R}) |\partial_{r} u|^{2} - \frac{1}{2 R^{2}} \int \phi''(\frac{r}{R}) |u|^{2} + 2 \int \psi(\frac{r}{R}) |u|^{2} (\frac{m + A_{\theta}[u]}{r})^{2} \\ - \int \phi(\frac{r}{R}) |u|^{4} - \frac{1}{2} \int [\psi(\frac{r}{R}) - \phi(\frac{r}{R})] |u|^{4}.
\endaligned
\end{equation}

Integrating by parts,
\begin{equation}\label{2.24.1}
\int \phi(\frac{r}{R}) |\partial_{r} u|^{2} = \int |\partial_{r}(\chi(\frac{r}{R}) u)|^{2} + O(\int_{r \geq R} \frac{1}{R^{2}} |u|^{2}).
\end{equation}
Therefore,
\begin{equation}\label{2.24.2}
\aligned
(\ref{2.24}) = 2 \int |\partial_{r}(\chi(\frac{r}{R}) u)|^{2} + 2 \int |\chi(\frac{r}{R}) u|^{2} (\frac{m + A_{\theta}[\chi(\frac{r}{R}) u]}{r})^{2} - \int |\chi(\frac{r}{R}) u|^{4} \\
+ O(\int_{r \geq R} \frac{1}{R^{2}} |u|^{2}) + 2 \int (\psi(\frac{r}{R}) - \phi(\frac{r}{R})) (\frac{m + A_{\theta}[u]}{r})^{2} |u|^{2} + \int |\chi(\frac{r}{R}) u|^{2} \cdot \{ (\frac{m + A_{\theta}[u]}{r})^{2} - (\frac{m + A_{\theta}[\chi u]}{r})^{2} \} \\
- \int [\chi(\frac{r}{R})^{2} - \chi(\frac{r}{R})^{4}] |u|^{4} - \frac{1}{2} \int [\psi(\frac{r}{R}) - \phi(\frac{r}{R})] |u|^{4}.
\endaligned
\end{equation}
Using $(\ref{1.3})$,
\begin{equation}\label{2.24.3}
2 \int |\partial_{r}(\chi(\frac{r}{R}) u)|^{2} + 2 \int |\chi(\frac{r}{R}) u|^{2} (\frac{m + A_{\theta}[\chi(\frac{r}{R}) u]}{r})^{2} - \int |\chi(\frac{r}{R}) u|^{4} = 4 E[\chi_{R} u].
\end{equation}
Next, since $A_{\theta}[u] \lesssim M^{2}$,
\begin{equation}\label{2.24.4}
O(\int_{r \geq R} \frac{1}{R^{2}} |u|^{2}) + 2 \int (\psi(\frac{r}{R}) - \phi(\frac{r}{R})) (\frac{m + A_{\theta}[u]}{r})^{2} |u|^{2} \lesssim \frac{1}{R^{2}} \int_{r \geq R} |u|^{2}.
\end{equation}
By direct computation,
\begin{equation}\label{2.24.5}
|A_{\theta}[u] - A_{\theta}[\chi_{R} u]| = \int_{0}^{r} [|u|^{2} - |\chi_{R} u|^{2}] r' dr',
\end{equation}
so $|A_{\theta}[u] - A_{\theta}[\chi_{R} u]|$ is supported on $r \geq R$. Therefore,
\begin{equation}\label{2.24.6}
\int |\chi(\frac{r}{R}) u|^{2} \cdot \{ (\frac{m + A_{\theta}[u]}{r})^{2} - (\frac{m + A_{\theta}[\chi_{R} u]}{r})^{2} \} \lesssim \frac{1}{R^{2}} \int_{r \geq R} |u|^{2}.
\end{equation}
Finally, for
\begin{equation}\label{2.24.7}
- \int [\chi(\frac{r}{R})^{2} - \chi(\frac{r}{R})^{4}] |u|^{4} - \frac{1}{2} \int [\psi(\frac{r}{R}) - \phi(\frac{r}{R})] |u|^{4},
\end{equation}
consider two cases separately, as in Lemma $\ref{l2.2}$. For $\| u \|_{\dot{H}^{1}} \lesssim 1$, by interpolation,
\begin{equation}\label{2.24.8}
- \int [\chi(\frac{r}{R})^{2} - \chi(\frac{r}{R})^{4}] |u|^{4} - \frac{1}{2} \int [\psi(\frac{r}{R}) - \phi(\frac{r}{R})] |u|^{4} \leq \int_{r \geq R} |u|^{4} \lesssim (\int_{r \geq R} |u|^{2}) \| u \|_{\dot{H}^{1}}^{2} \leq o_{R}(1),
\end{equation}
where $o_{R}(1)$ is a quantity that approaches $0$ as $R \rightarrow \infty$. For the last step, $(\ref{2.1})$ is used. When $\| u \|_{\dot{H}^{1}} \gg 1$, then by Proposition $\ref{p2.4}$,
\begin{equation}\label{2.24.9}
u = Q_{\lambda(t), \gamma(t)} + \epsilon(t, x) = \lambda(t) e^{i \gamma(t)} Q(\lambda(t) x) + \epsilon(t, x),
\end{equation}
with $\lambda(t) \gg 1$. Then by direct computation,
\begin{equation}\label{2.24.10}
\int_{r \geq R} \lambda(t)^{4} Q(\frac{x}{\lambda(t)})^{4} dx \leq o_{R}(1).
\end{equation}
Also, by Proposition $\ref{p2.4}$ and $(\ref{2.1})$,
\begin{equation}\label{2.24.11}
\int_{r \geq R} |\epsilon(t, x)|^{4} dx \lesssim \| \epsilon \|_{\dot{H}^{1}}^{2} \| \epsilon \|_{L^{2}(r \geq R)}^{2} \leq o_{R}(1).
\end{equation}

 Therefore,
\begin{equation}\label{2.25}
\frac{d}{dt} M(t) = E[\chi_{R} u] + o_{R}(1).
\end{equation}
Therefore,
\begin{equation}\label{2.26}
\int_{0}^{T} E[\chi_{R} u] dt \lesssim R + T o_{R}(1).
\end{equation}
Taking $R = T^{1/3}$, there exists a sequence $t_{n}' \rightarrow \infty$, $R_{n} \rightarrow \infty$, satisfying
\begin{equation}\label{2.27}
E[\chi(\frac{r}{R_{n}}) u(t_{n}')] \rightarrow 0.
\end{equation}
Therefore,
\begin{equation}\label{2.28}
\chi(\frac{r}{R_{n}}) u(t_{n}') = [Q + \epsilon]_{\lambda_{n}, \gamma_{n}},
\end{equation}
and for $n$ sufficiently large,
\begin{equation}\label{2.29}
\| \epsilon \|_{\dot{H}_{m}^{1}}^{2} \sim E[\chi(\frac{r}{R_{n}}) u(t_{n}')].
\end{equation}
Therefore, $\| \epsilon \|_{L^{2}(|x| \leq R_{n}')} \rightarrow 0$ for some $R_{n}' \rightarrow \infty$, and
\begin{equation}\label{2.30}
\| u \|_{L^{2}(|x| \geq R_{n}')} \rightarrow 0.
\end{equation}

The bounds $(\ref{2.1})$ on the mass implies that there exists $\lambda_{0} > 0$ such that $\lambda(t_{n}') \geq \lambda_{0} > 0$. Therefore,
\begin{equation}\label{2.31}
\| Q_{\lambda(t_{n}'), \gamma(t_{n}')} \|_{L^{2}(|x| \leq R_{n}')} \rightarrow \| Q \|_{L^{2}},
\end{equation}
as $n \rightarrow \infty$.
\end{proof}

\begin{remark}
Note that the estimate in $(\ref{2.26})$ utilizes that $T \rightarrow \infty$. We cannot use this argument if $u$ blows up in finite time in both directions.
\end{remark}

\section{Rigidity for finite time blowup solutions at the soliton}
Now we duplicate the result of \cite{merle1993determination} for the Chern--Simons--Schr{\"o}dinger equation, showing that if $u$ is a blowup solution to $(\ref{1.1})$ with $\| u_{0} \|_{L^{2}} = \| Q \|_{L^{2}}$ and $u_{0} \in H^{1}$, the soluton $u$ must be of the form $(\ref{1.7})$. Such a solution would violate $(\ref{2.1})$ in the scattering time direction.

\begin{theorem}\label{t3.1}
If $u$ is a finite time blowup solution with $\| u_{0} \|_{L^{2}} = \| Q \|_{L^{2}}$ and $u_{0} \in H^{1}$, then $u$ is equal to a pseudoconformal transformation of a soliton.
\end{theorem}
\begin{proof}
By time translation symmetry and the scaling symmetry, suppose that $u$ blows up at time $t = 0$, and let $u_{0}$ be the data for $u(t, r)$ at $t = -1$.

\begin{lemma}\label{l3.2}
Fix $R > 0$ large. Let $\phi \in C_{0}^{\infty}(\mathbb{R}^{2})$ be a smooth cut-off, $\phi(x) = 1$ for $|x| \leq 1$, and $\phi(x) = 0$ for $|x| > 2$. Then,
\begin{equation}\label{3.1}
\lim_{t \nearrow 0} \| \phi(\frac{x}{R}) |x| u(t, x) \|_{L^{2}} = 0.
\end{equation}
\end{lemma}
\begin{proof}
If $u$ blows up in finite time, $\lim_{t \nearrow 0} \| u(t) \|_{\dot{H}^{1}} = \infty$. Now let 
\begin{equation}\label{3.2}
\lambda(t) = \frac{\| u \|_{\dot{H}^{1}}}{\| Q \|_{\dot{H}^{1}}},
\end{equation}
and let
\begin{equation}\label{3.3}
v(t,x) = \frac{1}{\lambda(t)} u(t, \frac{x}{\lambda(t)}).
\end{equation}
Then by $(\ref{2.9})$, $E(v(t)) \rightarrow 0$, and $\| v(t) \|_{\dot{H}^{1}}$ and $\| v(t) \|_{L^{2}}$ are uniformly bounded for $-1 < t < 0$. Therefore, by Proposition $\ref{p2.3}$, for $t$ sufficiently close to $0$,
\begin{equation}\label{3.4}
v(t) = [Q + \epsilon]_{\tilde{\lambda}(t), \tilde{\gamma}(t)},
\end{equation}
and furthermore,
\begin{equation}\label{3.5}
\| \epsilon \|_{\dot{H}_{m}^{1}} \rightarrow 0.
\end{equation}
By $(\ref{3.4})$, $(\ref{3.5})$, and $\| v \|_{\dot{H}^{1}} = 1$, $\tilde{\lambda}(t) \sim 1$, so
\begin{equation}\label{3.6}
e^{-i \tilde{\gamma}(t)} v(t) \rightharpoonup Q, \qquad \text{in} \qquad L^{2},
\end{equation}
and since $\| v \|_{L^{2}} = \| Q \|_{L^{2}}$, $(\ref{3.6})$ can be upgraded to convergence in $L^{2}$. Since $\lambda(t) \nearrow \infty$ as $t \nearrow 0$, $(\ref{3.1})$ holds.
\end{proof}

\begin{lemma}\label{l3.3}
For any $R > 0$,
\begin{equation}\label{3.7}
\lim_{t \nearrow 0} \int \phi(\frac{x}{R}) Im(\bar{u} \cdot r \partial_{r} u) dx = 0.
\end{equation}
\end{lemma}
\begin{proof}
The argument is identical to the argument proving Lemma $\ref{l2.2}$, except that now, insert $\| \epsilon(t) \|_{L^{2}} \rightarrow 0$ into $(\ref{2.17})$--$(\ref{2.19})$, proving $(\ref{3.7})$.
\end{proof}

Returning to the proof of Theorem $\ref{t1.2}$, by direct computation,
\begin{equation}\label{3.8}
\frac{d}{dt} \int r^{2} |u|^{2} dx = 4 \int Im(\bar{u} \cdot r \partial_{r} u) dx.
\end{equation}
Integrating by parts,
\begin{equation}\label{3.9}
\frac{d}{dt} \int \phi(\frac{r}{R})^{2} r^{2} |u|^{2} dx = 4 \int \phi^{2}(\frac{x}{R}) Im(\bar{u} \cdot r \partial_{r} u) dx + \frac{8}{R} \int \phi(\frac{r}{R}) \phi'(\frac{r}{R}) r Im(\bar{u} \cdot r \partial_{r} u) dx.
\end{equation}
Also, by direct computation and integrating by parts,
\begin{equation}\label{3.10}
\aligned
\frac{d}{dt} \int \phi(\frac{r}{R})^{2} Im(\bar{u} \cdot r \partial_{r} u) dx = 2 \int \phi(\frac{r}{R})^{2} [|\partial_{r} u|^{2} + (\frac{m + A_{\theta}[u]}{r})^{2} |u|^{2} - \frac{1}{2} |u|^{4}]  dx \\ + \frac{C}{R} \int \phi'(\frac{r}{R}) \phi(\frac{r}{R}) \{ r |u_{r}|^{2} + \frac{1}{r} |u|^{2} \} dx.
\endaligned
\end{equation}
Therefore, by Lemmas $\ref{l3.2}$ and $\ref{l3.3}$, taking $R \nearrow \infty$, for any $-1 < t < 0$,
\begin{equation}\label{3.11}
\int |x|^{2} |u(t,x)|^{2} dx = 8 E[u] t^{2}.
\end{equation}
Making a pseudoconformal transformation of the solution, let
\begin{equation}\label{3.12}
v(t, r) = \frac{1}{|t|} u(-\frac{1}{t}, \frac{r}{|t|}) e^{i r^{2}/4t}.
\end{equation}
By $(\ref{3.12})$,
\begin{equation}\label{3.13}
\frac{1}{2} \int |\partial_{r} v|^{2} dx = \frac{1}{2t^{2}} \| u_{r}(-\frac{1}{t}, \cdot) \|_{L^{2}}, \qquad \frac{1}{2} \int \frac{r^{2}}{4t^{2}} |u(-\frac{1}{t}, \frac{r}{t})|^{2} r dr = \frac{1}{t^{2}} E[u],
\end{equation}
and by $(\ref{3.9})$,
\begin{equation}\label{3.14}
Re(\int \frac{1}{t^{2}} \overline{u_{r}(-\frac{1}{t}, \frac{r}{t})} \cdot \frac{ir}{2t^{2}} u(-\frac{1}{t}, \frac{r}{t}) dx) = -\frac{2}{t^{2}} E[u].
\end{equation}
Therefore, $E[v] = 0$, and thus, $v$ is a soliton, so $u$ is a pseudoconformal transformation of a soliton.
\end{proof}

\section{The Liouville theorem}
Now we have proved that the only solution to $(\ref{1.1})$ that satisfies $(\ref{2.1})$ has mass $\| u_{0} \|_{L^{2}} = \| Q \|_{L^{2}}$ and is global in both time directions. Then we complete the proof of the Liouville theorem by showing that $u$ is a soliton.

\begin{theorem}\label{t4.1}
The solution satisfying $\| u \|_{L^{2}} = \| Q \|_{L^{2}}$ and $(\ref{1.12})$ is the soliton.
\end{theorem}
\begin{proof}
We again use the virial identity in $(\ref{2.23})$ and $(\ref{2.24})$, only this time we integrate from $-T$ to $T$. Integrating by parts,
\begin{equation}\label{4.1}
\int \psi(\frac{r}{R}) Re[\bar{u} \cdot r \partial_{r} \Delta u] - \int \psi(\frac{r}{R}) Re[\Delta \bar{u} \cdot r \partial_{r} u] = 2 \int \phi(\frac{r}{R}) |\partial_{r} u|^{2} + O(\frac{1}{R^{2}} \int_{r \geq R} |u|^{2}),
\end{equation}
where again $\phi(r) = \partial_{r} (r \psi(r))$ and $\phi(r) = \chi(r)^{2}$ for some $\chi \in C_{0}^{\infty}$. Now then,
\begin{equation}\label{4.2}
\frac{1}{R^{2}} \int_{r \geq R} \lambda^{2} Q(\frac{x}{\lambda})^{2} dx = \frac{\lambda}{R^{2}} \int_{R}^{\infty} \frac{(\lambda r)^{2m + 1}}{(1 + (\lambda r)^{2m + 2})^{2}} dr \lesssim \frac{1}{R^{4} \lambda^{2}}.
\end{equation}
Next, for
\begin{equation}\label{4.3}
u = \lambda Q(\lambda x) + \lambda \epsilon(\lambda x),
\end{equation}
since $\| u \|_{L^{2}} = \| Q \|_{L^{2}}$, by standard linear algebra,
\begin{equation}\label{4.4}
\| \epsilon \|_{L^{2}}^{2} = -2\langle \lambda Q(\lambda r), \lambda \epsilon(\lambda r) \rangle.
\end{equation}
Therefore,
\begin{equation}\label{4.5}
\| \epsilon \|_{L^{2}}^{2} = 2 |\langle \epsilon, Q \rangle| \leq 2|\langle \chi(\frac{r}{R}) \lambda \epsilon(\lambda x), \lambda Q(\lambda x) \rangle| + 2|\langle (1 - \chi(\frac{r}{R})) \lambda \epsilon(\lambda x), \lambda Q(\lambda x)\rangle|.
\end{equation}
By H{\"o}lder's inequality,
\begin{equation}\label{4.6}
2|\langle (1 - \chi(\frac{r}{R})) \lambda \epsilon(\lambda x), \lambda Q(\lambda x)\rangle| \lesssim \frac{1}{R^{2} \lambda^{2}} \| \epsilon \|_{L^{2}}.
\end{equation}
Now let
\begin{equation}\label{4.7}
\tilde{Q}(r) = -\int_{r}^{\infty} Q(r) dr.
\end{equation}
Since $Q(r) \lesssim \frac{1}{r^{4}}$ for $r$ large and $Q(r) \leq 1$ for all $r$, $\tilde{Q}(r) \in L^{2}(\mathbb{R}^{2})$. Moreover,
\begin{equation}\label{4.8}
\partial_{r} \tilde{Q}(r) = Q(r), \qquad \text{and} \qquad \partial_{r} \tilde{Q}(\lambda r) = \lambda Q(\lambda r).
\end{equation}
Therefore,
\begin{equation}\label{4.9}
\langle \chi(\frac{r}{R}) \lambda \epsilon(\lambda r), \lambda Q(\lambda r) \rangle = \langle \chi(\frac{r}{R}) \lambda \epsilon(\lambda r), \partial_{r} \tilde{Q}(\lambda r) \rangle.
\end{equation}
Integrating by parts,
\begin{equation}\label{4.10}
(\ref{4.9}) \lesssim \frac{1}{\lambda} \| \chi(\frac{r}{R}) \lambda \epsilon(\lambda r) \|_{\dot{H}_{m}^{1}} \lesssim \frac{1}{\lambda} E[\chi(\frac{r}{R}) \lambda \epsilon(\lambda r) + \lambda Q(\lambda r)]^{1/2} \lesssim \frac{1}{\lambda} E[\chi(\frac{r}{R}) u]^{1/2} + \frac{1}{\lambda^{3} R^{3}}.
\end{equation}
Therefore,
\begin{equation}\label{4.11}
\| \epsilon \|_{L^{2}}^{2} \lesssim \frac{1}{\lambda} E[\chi(\frac{r}{R}) u]^{1/2} + \frac{1}{\lambda^{3} R^{3}},
\end{equation}
and by $(\ref{4.2})$ and $(\ref{4.11})$,
\begin{equation}\label{4.12}
\int_{r \geq R} \frac{1}{R^{2}} |u|^{2} \lesssim \frac{1}{\lambda R^{2}} E[\chi(\frac{r}{R}) u]^{1/2} + \frac{1}{\lambda^{2} R^{4}}.
\end{equation}
Also by $(\ref{4.2})$,
\begin{equation}\label{4.13}
\int_{r \geq R} \frac{R}{r} |u|^{4} dx \lesssim \frac{1}{\lambda^{2} R^{4}} + \int_{r \geq R} \frac{R}{r} |\lambda \epsilon(\lambda x)|^{4} dx.
\end{equation}
By standard perturbation theory and the fact that the $L_{t,x}^{4}$ norm is invariant under the scaling $(\ref{1.12})$,
\begin{equation}\label{4.14}
\int_{t_{0}}^{t_{0} + \frac{1}{\lambda(t_{0})^{2}}} \int |\lambda \epsilon(t, \lambda x)|^{4} dx dt \lesssim \| \epsilon(t_{0}) \|_{L^{2}}^{4} \lesssim \lambda(t_{0})^{2} \int_{t_{0}}^{t_{0} + \frac{1}{\lambda(t_{0})^{2}}} \| \epsilon(t) \|_{L^{2}}^{4}.
\end{equation}
Plugging in $(\ref{4.11})$, since $\lambda(t) \sim \lambda(t_{0})$ for $t \in [t_{0}, t_{0} + \frac{1}{\lambda(t_{0})^{2}}]$,
\begin{equation}\label{4.15}
\lambda(t_{0})^{2} \int_{t_{0}}^{t_{0} + \frac{1}{\lambda(t_{0})^{2}}} \| \epsilon(t) \|_{L^{2}}^{4} \lesssim  \int_{t_{0}}^{t_{0} + \frac{1}{\lambda(t_{0})^{2}}} E[\chi(\frac{r}{R}) u] dt + \int_{t_{0}}^{t_{0} + \frac{1}{\lambda(t_{0})^{2}}} \frac{1}{R^{4}} dt.
\end{equation}
Making an averaging argument, for a fixed $R_{0}$,
\begin{equation}
\sum_{j} \int_{r \geq 2^{j} R_{0}} \frac{2^{j} R_{0}}{r} |f(x)|^{4} dx \lesssim \| f(x) \|_{L^{4}}^{4}.
\end{equation}
Therefore, for any $\delta > 0$ there exists some $R_{0} \leq R_{\ast} \leq C(\delta) R_{0}$ such that
\begin{equation}\label{4.16}
\aligned
 \int_{t_{0}}^{t_{0} + \frac{1}{\lambda(t_{0})^{2}}} \int_{r \geq R_{\ast}} \frac{R_{\ast}}{r} |\lambda \epsilon(t, \lambda x)|^{4} dx dt \lesssim \delta \int_{t_{0}}^{t_{0} + \frac{1}{\lambda(t_{0})^{2}}} \| \epsilon(t) \|_{L^{4}}^{4} dt \\ \lesssim \delta \int_{t_{0}}^{t_{0} + \frac{1}{\lambda(t_{0})^{2}}} E[\chi(\frac{r}{R_{\ast}}) u] dt + \int_{t_{0}}^{t_{0} + \frac{1}{\lambda(t_{0})^{2}}} \frac{1}{R^{4}} dt.
 \endaligned
\end{equation}
Therefore, we have proved
\begin{equation}\label{4.17}
\int_{a}^{b} E[\chi(\frac{r}{R}) u] dt \lesssim R \| \epsilon(a) \|_{L^{2}} + R \| \epsilon(b) \|_{L^{2}} + \int_{a}^{b} \frac{1}{R^{4}} dt.
\end{equation}
Taking $R = T^{1/4}$, averaging $(\ref{4.17})$, and plugging in $(\ref{4.11})$, we have proved
\begin{equation}\label{4.18}
\lim_{T \nearrow \infty} \inf_{t \in [0, T]} \| \epsilon(t) \|_{L^{2}} = 0, \qquad \lim_{T \nearrow \infty} \inf_{t \in [-T, 0]} \| \epsilon(t) \|_{L^{2}} = 0.
\end{equation}
Now for any $j \in \mathbb{Z}$, $j \geq 0$, let
\begin{equation}\label{4.19}
t_{j}^{+} = \inf \{ t \geq 0 : \| \epsilon(t) \|_{L^{2}} = 2^{-j} \}, \qquad t_{j}^{-} = \sum \{ t < 0 : \| \epsilon(t) \|_{L^{2}} = 2^{-j} \}.
\end{equation}
Then let $T_{j} = t_{j}^{+} - t_{j}^{-}$, $I_{j} = [t_{j}^{-}, t_{j}^{+}]$. Then $T_{j} \rightarrow \infty$ as $j \rightarrow \infty$. Then by $(\ref{4.11})$,
\begin{equation}\label{4.20}
\| \epsilon(t_{j}^{\pm}) \|_{L^{2}}^{4} \lesssim \frac{1}{\lambda(t_{j}^{\pm})^{2}} E[\chi(\frac{r}{R}) u(t_{j}^{\pm})] + \frac{1}{R_{j}^{6}}.
\end{equation}
Therefore, by $(\ref{4.17})$,
\begin{equation}\label{4.21}
\int_{I_{j}} E[\chi(\frac{r}{R}) u] dt \lesssim \frac{R_{j}}{T_{j}} (\int_{I_{j}} E[\chi(\frac{r}{R}) u]^{1/4} dt) + \frac{T_{j}}{R_{j}^{4}}.
\end{equation}
Therefore,
\begin{equation}\label{4.22}
\int_{I_{j}} E[\chi(\frac{r}{R}) u] dt \lesssim \frac{R_{j}}{T_{j}^{1/4}} + \frac{T_{j}}{R_{j}^{4}} \sim 1.
\end{equation}
Then, by the dominated convergence theorem, taking $j \rightarrow \infty$,
\begin{equation}\label{4.23}
\int_{\mathbb{R}} E[u] dt \lesssim 1,
\end{equation}
which implies $E[u] \equiv 0$, and thus $u$ is a soliton.

\end{proof}

\section{Acknowledgements}
During the writing of this paper, the author was partially supported by NSF grant DMS-2153750.

\bibliography{biblio}
\bibliographystyle{alpha}
\end{document}